\documentclass[12pt]{amsart}
\usepackage[normalem]{ulem}
\usepackage{amsmath}
\usepackage{hyperref}
\usepackage{amssymb}
\usepackage{mathrsfs}
\usepackage{tikz}
\tikzstyle{vertex}=[circle]
\tikzstyle{goto}=[->,shorten >=1pt,>=stealth,semithick]
\usetikzlibrary{graphs}
\usepackage{comment}

\allowdisplaybreaks[2]

\newtheorem{thm}{Theorem}[section]

\theoremstyle{definition}

\theoremstyle{remark}

\numberwithin{equation}{section}

\newcommand{\CC}{\mathbb{C}}

\newcommand{\Q}{\mathcal{Q}}

\newcommand{\Zz}{\mathbb{Z}}
\newcommand{\Tz}{\mathbb{T}}

\newcommand{\iso}{\text{Iso}}

\newcommand{\cG}{\mathcal{G}}

\newcommand{\Et}{\widehat E_{\text{tight}}}

\newcommand{\gt}{\mathcal{G}_{\text{tight}}}

\newcommand{\vp}{\varphi}

\newcommand{\cF}{\mathcal{F}}
\newcommand{\tight}{\textup{tight}}
\newcommand{\Iso}{\textup{Iso}}

\begin{document}

\title[Corrigendum]{Corrigendum to ``A New Uniqueness Theorem for the Tight C*-algebra of an Inverse Semigroup'' [C. R. Math. Acad. Sci. Soc. R. Can. 44 (2022), no. 4, 88--112] }

\author{Chris Bruce} 
\address[C. Bruce]{School of Mathematics and Statistics, University of Glasgow, University Place, Glasgow G12 8QQ, United Kingdom}
\email{Chris.Bruce@glasgow.ac.uk}

\thanks{C. Bruce has received funding from the European Union’s Horizon 2020 research and innovation programme under the Marie Sklodowska-Curie grant agreement No 101022531 and from the European Research Council (ERC) under the European Union’s Horizon 2020 research and innovation programme 817597.}

\author{Charles Starling}
\address[C. Starling]{Carleton University, School of Mathematics and Statistics. 4302 Herzberg Laboratories}
\email{cstar@math.carleton.ca}

\thanks{C. Starling is partially funded by NSERC and by an internal Carleton research grant.}

\begin{abstract}
We correct the proof of Theorem~4.1 from [C. R. Math. Acad. Sci. Soc. R. Can. \textbf{44} (2022), no. 4, 88--112]. 
\end{abstract}

\maketitle

\section{Introduction}
\label{Intro}
There is a flaw in the proof of \cite[Theorem~4.1]{St22}. To explain this, let us recall the setup for \cite[Theorem~4.1]{St22}. Let $P$ be a right LCM monoid, i.e., a left cancellative monoid such that for all $p,q\in P$, the intersection $pP\cap qP$ is either empty or of the form $rP$ for some $r\in P$. Let $S=\{[p,q] : p,q\in P\}\cup\{0\}$ be the inverse semigroup associated with $P$ in \cite[Proposition~3.2]{StLCM}. 
(If $P$ is not left reversible, then $S$ is isomorphic to the left inverse hull of $P$ via the map that sends $0$ to $0$ and sends $[p,q]$ to the partial bijection $qP\to pP$ given by $qx\mapsto px$.) 
Let
\[
S^{\Iso}:=\{[p,q] :paP\cap qaP\neq\emptyset\text{ for all }a\in P\}
\]
be the inverse semigroup from \cite[Section~4]{St22}, and denote by $\cG_{\tight}(S^\Iso)$ and $\cG_{\tight}(S)$ the tight groupoids of $S^\Iso$ and $S$, respectively (see \cite{Ex08} and \cite{EP}). The C*-algebra $C_r^*(\cG_{\tight}(S))$ is the reduced boundary quotient C*-algebra of $P$. Since $\cG_{\tight}(S^\Iso)$ is identified with an open subgroupoid of $\cG_{\tight}(S)$, there is a canonical inclusion of reduced groupoid C*-algebras $C_r^*(\cG_{\tight}(S^\Iso))\subseteq C_r^*(\cG_{\tight}(S))$. Explicitly, we have 
\[
C_r^*(\cG_{\tight}(S))=\overline{\textup{span}}(\{T_{[p,q]} : [p,q]\in S\}),
\]
where $T_{[p,q]}$ is the characteristic function of the compact open bisection $\Theta([p,q],D_{qP})$ (see \cite[Section~3]{St22} for this notation), and $C_r^*(\cG_{\tight}(S^\Iso))$ is identified with the C*-subalgebra generated by the partial isometries $T_{[p,q]}$ for $[p,q]\in S^\Iso$.

Assume that $S$ satisfies condition (H) from \cite[Definition~3.1]{St22}, and suppose $\pi\colon C_r^*(\cG_{\tight}(S))\to B$ is a representation in a C*-algebra $B$. Then, \cite[Theorem~3.4]{St22} says that $\pi$ is injective if and only if its restriction to $C_r^*(\cG_{\tight}(S^\Iso))$ is injective. It is asserted in the proof of \cite[Theorem~4.1]{St22} that to prove this latter claim, it suffices to prove that $\pi$ is injective on the dense *-subalgebra $A_0:=\textup{span}(\{T_{[p,q]} : [p,q]\in S^\Iso\})$. However, this assertion is false: If we take $P=\Zz$, then $C_r^*(\cG_{\tight}(S^\Iso))=C_r^*(\cG_{\tight}(S))\cong C^*(\Zz)$, and under the canonical isomorphism $C^*(\Zz)\cong C(\Tz)$, $A_0$ is carried onto the *-subalgebra of Laurent polynomials in $C(\Tz)$. Given any infinite, proper compact subset $K\subseteq \Tz$, the map $C(\Tz)\to C(K)$ given by $f\mapsto f\vert_K$ is a non-injective *-homomorphism that is injective on $A_0$.

\section{The proof of \cite[Theorem~4.1]{St22}}

We give a proof of \cite[Theorem~4.1]{St22}. We shall use the notation from \cite{St22} freely. The core submonoid of $P$ is
\[
P_c:=\{p\in P : pP\cap qP\neq\emptyset\text{ for all } q\in P\},
\]
with associated inverse semigroup 
\[
S_c:=\{[p,q] : p,q\in P_c\}.
\]

\cite[Theorem~4.1]{St22} is stated with the assumption that the full and reduced groupoid C*-algebras of the tight groupoid $\gt(S)$ coincide. It is clear that the following version stated for reduced groupoid C*-algebras implies \cite[Theorem~4.1]{St22}.

\begin{thm}
\label{thm}
Let $P$ be a right LCM monoid and $S$ the associated inverse semigroup as in \cite[Proposition~3.2]{StLCM}, let $\Q_r(P)= C^*_r(\gt(S))$ denote its reduced boundary quotient C*-algebra, and let $\Q_{r,c}(P) = C^*(T_{[p,q]}: p,q\in P_c)\subseteq \Q_r(P)$ be the C*-subalgebra generated by the core submonoid. Suppose that $S$ satisfies condition (H) from \cite[Definition~3.1]{St22}. Then, a $*$-homomorphism $\pi: \Q_r(P)\to B$ to a C*-algebra $B$ is injective if and only if it is injective on $\Q_{r,c}(P)$. 	
\end{thm}
\begin{proof}
	Let $\pi:\Q_r(P)\to B$ be a $*$-homomorphism that is injective on $\Q_{r,c}(P)$. We wish to show that $\pi$ is injective, and by \cite[Theorem~3.4]{St22} it is enough to show that $\pi$ is injective on $A:= C^*(T_{[p,q]}: [p,q]\in S^\iso)$. If $P$ is left reversible, then $P=P_c$ and there is nothing to prove, so assume $P$ is not left reversible.
Since $S$ satisfies (H) from \cite[Definition~3.1]{St22}, the groupoid $\gt(S)$ is Hausdorff by \cite[Theorem~3.16]{EP}. Thus, we have a canonical faithful conditional expectation $E\colon \Q_r(P)\to C(\Et)$. We follow the strategy of the proof of \cite[Theorem~5.1]{LS22} and will show that there is a linear map $\vp$ defined on $\pi(A)$ such that $\vp\circ \pi(a) = \pi \circ E(a)$ for every $a\in A$. One can see that this amounts to showing that
	\[
	\pi(a) \mapsto \pi(E(a)), \hspace{1cm}a\in A
	\]
	is well-defined. We will be done if we show that $\|\pi(a)\| \geq \|\pi(E(a))\|$ for all $a$ in the canonical dense subalgebra $A_0:=\textup{span}(T_{[p,q]}: [p,q]\in S^\iso)$ of $A$. 
	
Let $a = \sum_{f\in F}\lambda_fT_{[p_f, q_f]}\in A_0$ be a finite linear combination of the generators of $A$, where $F$ is a finite index set, $[p_f, q_f]\in S^\iso$, and $\lambda_f\in\CC$.  
 For each $f\in F$, let $r_f$ be an element of $P$ such that $p_fP\cap q_fP = r_fP$. 
	
As a function on $\gt(S)$, the element $a$ is a linear combination of characteristic functions on the compact open bisections $\Theta([p_f,q_f], D_{r_{f}P})$ for $f\in F$. Here, we used that $D_{q_fP}=D_{r_fP}$ by \cite[Lemma~4.2]{St22}. By \cite[Proposition~3.14]{EP}, we have $E(a)=\sum_{f\in F}\lambda_f1_{\cF_{[p_f, q_f]}}$, where $\cF_{[p_f, q_f]}$ is a certain compact open subset of $D_{r_fP}$ (we shall not need the precise definition of $\cF_{[p_f, q_f]}$ here; that $\cF_{[p_f, q_f]}$ is compact open suffices for our purposes).

Each nonempty subset $F'\subseteq F$ determines a compact open subset of $\Et$ given by \begin{equation*}
		U_{F'}:= \bigcap_{f\in F'} D_{r_fP}\setminus \left(\bigcup_{g\in F\setminus F'} D_{r_gP}\right).
\end{equation*}
 Since $\bigcup_{f\in F}D_{r_fP}=\bigsqcup_{\emptyset\neq F'\subseteq F}U_{F'}$ and $\cF_{[p_f, q_f]}\subseteq D_{r_fP}$ for all $f\in F$, the support of $E(a)$ is contained in $\bigsqcup_{\emptyset\neq F'\subseteq F}U_{F'}$. Thus, there exists $\emptyset\neq F'\subseteq F$ such that 
 \[
 \|E(a)\|=\max\{|a(u)| : u \in U_{F'}\}.
 \]
 If $E(a)=0$, then clearly $\|\pi(E(a))\|\leq \|\pi(a)\|$, so we may assume $E(a)\neq 0$, in which case $U_{F'}$ is nonempty. Since the ultrafilters are dense in $\Et$ and $E(a)=a\vert_{\Et}$ takes on only finitely many values on $\Et$, we can find an ultrafilter $\xi\in U_{F'}$ such that $\|E(a)\|=|a(\xi)|$.
	
Note that $r_fP\in \xi$ for all $f\in F'$ and $r_gP\not\in\xi$ for all $g\in F\setminus F'$. Since $P$ is not left reversible and $\xi$ is an ultrafilter, for each $g\in F\setminus F'$ we can find $k_g\in P$ such that $k_gP\in\xi$ and $k_gP\cap r_gP = \emptyset$ (see \cite[Lemma~12.3]{Ex08}). Since $\xi$ is a filter, we have 
	\[
 \bigcap_{f\in F'}r_fP\cap \bigcap_{g\in F\setminus F'}k_gP\neq \emptyset,
 \]
and this intersection must be of the form $bP$ for some $b\in P$ with $bP\in\xi$. Moreover, we have $bP\subseteq r_fP$ for all $f\in F'$ and $bP\cap r_gP = \emptyset$ for all $g\in F\setminus F'$, so that $D_{bP}\subseteq U_{F'}$. Since $\|E(a)\|=|E(a)(\xi)|$, $\xi\in D_{bP}$, and $T_{[b,b]}=1_{D_{bP}}$, we have $\|E(a)\| =\|T_{[b,b]}E(a)\|$. Moreover, since $bP\cap r_gP=\emptyset$ for all $g\in F\setminus F'$, we have for $\eta\in D_{bP}$ that $E(T_{[p_f,q_f]})(\eta)=0$ unless $f\in F'$ (note that $E(T_{[p_f,q_f]})$ has support in $D_{r_fP}$ by \cite[Proposition~3.14]{EP}). Hence,
	\begin{equation}
     \label{eqn:norm}
	\|E(a)\| =\|T_{[b,b]}E(a)\|=\sup_{\eta\in D_{bP}}|T_{[b,b]}(\eta)E(a)(\eta)|=\left\|T_{[b,b]}E\left(\sum_{f\in F'}\lambda_fT_{[p_f,q_f]}\right)T_{[b,b]}\right\|. 
	 \end{equation} 
	
Now \cite[Lemma~4.2]{St22} implies $T_{[b,1]}^*T_{[p_g,q_g]}T_{[b,1]} = 0$ for all $g\in F\setminus F'$, while \cite[Lemma~4.3]{St22} implies that $T_{[b,1]}^*T_{[p_f,q_f]}T_{[b,1]} \in \Q_{r,c}(P)$ for all $f\in F'$. 
We have $\ker(\pi)\cap C(\Et)=C_0(U)$, where $U\subseteq \Et$ is an open invariant subset. Since $1_{\Et}\in\Q_{r,c}(P)$, $\pi(1_{\Et})\neq 0$, so that $U$ is a proper subset of $\Et$. The groupoid $\gt(S)$ is minimal by \cite[Lemma~4.2]{StLCM}, so $U$ must be empty. Thus, $\pi$ is injective -- and hence isometric -- on $C(\Et)$, so that $\|E(a)\| = \|\pi(E(a))\|$.  
Thus, we can make the following estimate:
		
\begin{align*}
	\|\pi(a)\| &= \left\|\pi\left(\sum_{f\in F}\lambda_fT_{[p_f, q_f]}\right)\right\|\\
	           &\geq \left\|\pi(T_{[b,1]}^*)\pi\left(\sum_{f\in F}\lambda_fT_{[p_f, q_f]}\right)\pi(T_{[b,1]})\right\|&\text{submultiplicativity, $\pi(T_{[b,1]})$ an isometry}\\
	           &= \left\|\pi\left(\sum_{f\in F'}\lambda_fT_{[b,1]}^*T_{[p_f, q_f]}T_{[b,1]}\right)\right\|&\text{by choice of $b$}\\
	           & = \left\|\sum_{f\in F'}\lambda_fT_{[b,1]}^*T_{[p_f, q_f]}T_{[b,1]}\right\|&\text{$\pi$ is isometric on $\Q_{r,c}(P)$}\\
	           &=  \left\|\sum_{f\in F'}\lambda_fT_{[b,b]}T_{[p_f, q_f]}T_{[b,b]}\right\|&\text{submultiplicativity, $T_{[b,1]}$ an isometry}\\
	           &\geq  \left\|E\left(\sum_{f\in F'}\lambda_fT_{[b,b]}T_{[p_f, q_f]}T_{[b,b]}\right)\right\|&\text{$E$ is contractive}\\
            &=\left\|T_{[b,b]}E\left(\sum_{f\in F'}\lambda_fT_{[p_f, q_f]})\right)T_{[b,b]}\right\|&\text{$T_{[b,b]}$ is in the multiplicative domain of $E$}\\
	           & = \|E(a)\| & \text{by \eqref{eqn:norm}}\\
	           & = \|\pi(E(a))\| & \text{$\pi$ is isometric on $C(\Et)$}.
\end{align*}

Thus, the map $\pi(a) \mapsto \pi(E(a))$ is a well-defined linear idempotent contraction on the dense *-subalgebra $\pi(A_0)$ of $\pi(A)$, so it extends to a linear map on $\pi(A)$.  

To complete the proof, suppose that $\pi(x) = 0$ for some $x\in A$. Then, $\pi(x^*x) = 0$, implying $\pi(E(x^*x)) = 0$. Since $\pi$ is faithful on the image of $E$ we must have $E(x^*x) = 0$, and since $E$ is faithful we get $x^*x = 0$, implying $x=0$.  
\end{proof}

\section*{Acknowledgement} C. Bruce would like to thank Kevin Aguyar Brix for several helpful conversations and Adam Dor-On for comments on a draft version. C. Starling thanks Ruy Exel and Marcelo Laca for valuable insights.

\end{document}